\documentclass[a4paper,10pt,twoside,reqno]{amsart}

\usepackage[utf8]{inputenc}
\usepackage{amsmath, amsfonts, amssymb,amsthm}
\usepackage[mathscr]{eucal}
\usepackage[pdftex]{graphicx}
\usepackage{booktabs}
\usepackage{hyperref}
\usepackage{enumerate}
\usepackage[T1]{fontenc}
\usepackage[sc]{mathpazo}

\newtheorem{theorem}{Theorem}[section]
\newtheorem{lemma}{Lemma}[section]

\theoremstyle{remark}
  \newtheorem{remark}{Remark}[section]

\numberwithin{equation}{section}

\DeclareMathOperator{\arctanh}{arctanh}
\DeclareMathOperator{\arccosh}{arccosh}

\title[Compact embedded $H$-surfaces in $\mathbb{S}^2\times\mathbb{R}$]{Compact embedded surfaces with constant mean curvature in $\mathbb{S}^2\times\mathbb{R}$.}

\author{Jos\'{e} M. Manzano}
\address{Instituto de Ciencias Matem\'{a}ticas -- CSIC\\
Madrid, Spain}
\email{manzanoprego@gmail.com}

\author{Francisco Torralbo}
\address{Departamento de Matem\'{a}ticas\\
Universidad de C\'{a}diz, Spain}
\email{francisco.torralbo@uca.es}

\thanks{The authors are partially supported by the Spanish MINECO-Feder research projects MTM2014-52368-P and MTM2017-89677-P. The first author is also supported by the ICMAT–Severo Ochoa grant SEV-2015-0554.}

\subjclass[2010]{Primary 53A10; Secondary 53C30}

\keywords{Constant mean curvature surfaces, compact embedded surfaces, homogeneous three-manifolds, product spaces, conjugate constructions}

\begin{document}

\begin{abstract}
We obtain compact orientable embedded surfaces with constant mean curvature $0<H<\frac{1}{2}$ and arbitrary genus in $\mathbb{S}^2\times\mathbb{R}$. These surfaces have dihedral symmetry and desingularize a pair of spheres with mean curvature $\frac{1}{2}$ tangent along an equator. This is a particular case of a conjugate Plateau construction of doubly periodic surfaces with constant mean curvature in $\mathbb{S}^2\times\mathbb{R}$, $\mathbb{H}^2\times\mathbb{R}$, and $\mathbb{R}^3$ with bounded height and enjoying the symmetries of certain tessellations of $\mathbb{S}^2$, $\mathbb{H}^2$, and $\mathbb{R}^2$ by regular polygons.
\end{abstract}

\maketitle

\section{Introduction}

In the theory of constant mean curvature $H$ surfaces ($H$-surfaces in the sequel), the existence of compact embedded examples in different three-manifolds has been an appealing problem whose origins probably date back to the mid twentieth century and the Alexandrov problem. Although simply connected ambient manifolds with non-positive curvature usually admit very few examples, the moduli space of compact embedded $H$-surfaces when the ambient space is positively curved seems to be crowded. This can be illustrated by the recent result by Marques and Neves~\cite{MarNev} that any three-manifold with dimension at most seven and positive Ricci curvature admits an infinite number of compact embedded minimal hypersurfaces (i.e., with $H=0$). Assuming the ambient spaces to be simply connected makes it harder to find such minimal surfaces in the sense that minimizing area in a isotopy class is not a suitable method (see~\cite{MSY}).

We will focus on the simply connected three-manifold $\mathbb{S}^2\times\mathbb{R}$ with non-nega\-tive Ricci curvature, where we will overtake the problem of determining the possible topological types of compact embedded $H$-surfaces. Similar problems have been considered in the literature, e.g., Lawson~\cite{Lawson} showed that a compact surface can be minimally embedded in the three-sphere $\mathbb{S}^3$ if and only if it is orientable, and the second author extended this result to the larger family of Berger spheres~\cite{Tor2011}; moreover, the authors and Plehnert~\cite{MPT} showed that a compact surface can be minimally embedded in a Riemannian product $\mathbb{S}^2\times\mathbb{S}^1(r)$ if and only if it has even Euler characteristic. Although the only compact minimal surfaces immersed in $\mathbb{S}^2\times\mathbb{R}$ are the horizontal slices $\mathbb{S}^2\times\{t_0\}$, $t_0\in\mathbb{R}$, in the non-compact case Hoffman, Traizet and White~\cite{HTW} proved the existence of embedded minimal surfaces in $\mathbb{S}^2\times\mathbb{R}$ with two helicoidal ends and arbitrary genus.

On the contrary, the moduli space of compact embedded $H$-surfaces seems to be plentiful if $H>0$. Other than rotationally invariant spheres and tori, the authors~\cite{MT} provided the first non-trivial compact embedded $H$-surfaces, namely, a $1$-parameter family of $H$-tori resembling horizontal unduloids. If $0<H<\frac{1}{2}$, the first author proved in~\cite[Corollary~3.5]{Manzano} that any compact embedded $H$-surface in $\mathbb{S}^2\times\mathbb{R}$ must be a bigraph over a domain of $\mathbb{S}^2$ whose complement consists of finitely-many strictly convex topological disks, and here we will prove that such $H$-surfaces exist for all genera, i.e., for any number of disks. It is worth emphasizing the geometric content of the value $H=\frac{1}{2}$ in $\mathbb{S}^2\times\mathbb{R}$, for it has not been noticed probably in the literature yet. On the one hand, the convexity of the complement fails to hold in general if $H>\frac{1}{2}$ as shown by the horizontal unduloids in~\cite{MT}. On the other hand, $H=\frac{1}{2}$ is precisely the value such that rotationally invariant $H$-spheres are bigraphs over hemispheres. Since going over the equator leads to self-intersections that prevent the use of conjugate techniques and the existence of symmetric examples, we daresay (at least heuristically) that this value should be relevant in the classification of embedded $H$-surfaces.

Our existence result relies on a more general construction in the Riemannian product $\mathbb{M}^2(\epsilon)\times\mathbb{R}$, where $\mathbb{M}^2(\epsilon)$ stands for the simply connected complete surface with constant curvature $\epsilon\in\{-1,0,1\}$. Throughout the text, given integers $k,m\geq 2$ (not both equal to $2$), an $(m,k)$-tessellation of $\mathbb{M}^2(\epsilon)$ will be a tessellation of $\mathbb{M}^2(\epsilon)$ by regular $m$-gons such that $k$ of them meet at each vertex, see Figure~\ref{fig:tessellation}. The centers and vertexes of the $m$-gons will be called the centers and vertexes of the tessellation. An straightforward application of the Gau\ss-Bonnet formula to one of the polygons reveals that $\epsilon$ is uniquely determined as the sign of $\frac{1}{k}+\frac{1}{m}-\frac{1}{2}$. More explicitly,
\begin{itemize}
 	\item If $\frac{1}{k}+\frac{1}{m}>\frac{1}{2}$, then we get a tessellation of $\mathbb{S}^2$. There are two infinite families (the \emph{beach ball} with $m=2$ and the degenerated case $k=2$) plus five exceptional cases corresponding to the Platonic solids, see Table~\ref{table}.
 	\item If $\frac{1}{k}+\frac{1}{m}=\frac{1}{2}$, then $(m,k)\in\{(3,6),(4,4),(6,3)\}$ and we get the three tessellations of $\mathbb{R}^2$ corresponding to triangles, squares and hexagons. 
 	\item If $\frac{1}{k}+\frac{1}{m}<\frac{1}{2}$, then we get tessellations of $\mathbb{H}^2$.
 \end{itemize}
The following theorem ensures that we can find $H$-surfaces in $\mathbb{M}^2(\epsilon)\times\mathbb{R}$ with the symmetries of $(m,k)$-tessellations provided that $H$ is suitably bounded.

\begin{table}[bp]
\begin{center}
\begin{tabular}{cccc} \toprule
Initial tessellation&$(m,k)$&$\alpha(m,k)$&Genus of $\Sigma_{H,m,k}$ \\\midrule
\emph{Beach ball}&$(2,g+1)$&$1$&$g$\\
\emph{Degenerated}&$(g+1,2)$&$\cot^2(\frac{\pi}{2+2g})$&$1$\\\midrule
Tetrahedron&$(3,3)$&$2+\sqrt{3}$&$3$\\
Hexahedron&$(4,3)$&$5+2\sqrt{6}$&$5$\\
Octahedron&$(3,4$)&$3+2\sqrt{2}$&$7$\\
Dodecahedron&$(5,3)$&$8+4 \sqrt{3}+3 \sqrt{5}+2 \sqrt{15}$&$11$\\
Icosahedron&$(3,5)$&$4+\sqrt{5}+2 \sqrt{5+2 \sqrt{5}}$&$19$\\\bottomrule
\end{tabular}
\end{center}
\caption{$(m, k)$-tessellations and compact orientable surfaces embedded in $\mathbb{S}^2\times\mathbb{R}$ with $H< \alpha(m, k)$ and $g\geq 2$, see Theorem~\ref{thm:embeddedness}.}\label{table}
\end{table}

\begin{theorem}\label{thm:tessellations}
Let $k,m\geq 2$ be integers, and let $\epsilon\in\{-1,0,1\}$ be the sign of $\frac{1}{k}+\frac{1}{m}-\frac{1}{2}$. Given $H\in\mathbb{R}$ such that $0<4H^2<\alpha(m,k)$, where
\[\alpha(m,k)=\begin{cases}\frac{\sin(\frac{\pi}{k})+\cos(\frac{\pi}{m})}{\epsilon(\sin(\frac{\pi}{k})-\cos(\frac{\pi}{m}))}&\text{if }\epsilon\neq 0,\\\infty&\text{if }\epsilon=0,\end{cases}\]
there is a connected $H$-surface $\Sigma_{H,m,k}\subset\mathbb{M}^2(\epsilon)\times\mathbb{R}$ with the symmetries of a $(m,k)$-tessellation of $\mathbb{M}^2(\epsilon)$ and symmetric with respect to $\mathbb{M}^2(\epsilon)\times\{0\}$. 
Furthermore:
\begin{enumerate}
  \item[(a)] If $4H^2\to 0$, then $\Sigma_{H,m,k}$ converges to a double cover of $\mathbb{M}^2(\epsilon)\times\{0\}$ with singularities at the centers of the $(m,k)$-tessellation.

  \item[(b)] If $4H^2\to \alpha(m,k)$, then $\Sigma_{H,m,k}$ converges to a stack of tangent $H$-spheres of $\mathbb{M}^2(\epsilon)\times \mathbb{R}$ with centers at the vertexes of the $(m,k)$-tessellation.
\end{enumerate}
\end{theorem}

\begin{figure}[hbp]\label{fig:4-gon-tessellation}
\centering
\includegraphics{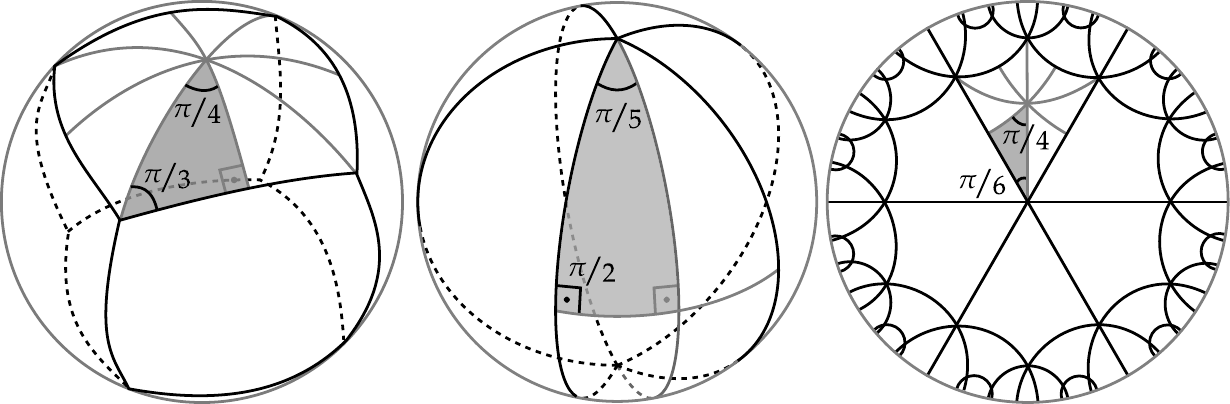}
\caption{A $(4,3)$-tessellations of $\mathbb{S}^{2}$ corresponding to the hexahedron (left) a $(2,5)$-tessellation (\emph{beach ball tessellation}, middle) and a $(4,6)$-tessellation of $ \mathbb{H}^2$ (right). The shaded regions represent the target triangles in our construction, see also Figure~\ref{fig:first-polygon}.}\label{fig:tessellation}
\end{figure}

Theorem~\ref{thm:tessellations} was already known in the Euclidean space. Lawson~\cite{Lawson} used a conjugation argument to show the existence of doubly periodic $H$-surfaces with bounded height in $\mathbb{R}^3$ over square and triangular lattices, which was revisited and extended by Gro\ss e-Brauckmann~\cite{KGB}. Note that $\alpha(m,k)=\infty$ is due to the fact that we can reach any value of the mean curvature by homotheties of $\mathbb{R}^3$. Therefore we will restrict our discussion to $\epsilon\in\{-1,1\}$, where homotheties fail to exist and the size of the $m$-gons is determined by $(m,k)$, but our proof also works when $\epsilon=0$. Note that a $(k,m)$-tessellation and an $(m,k)$-tessellation are \emph{dual} in the sense that they can be obtained from each other by swapping centers and vertexes, and hence they have the same isometry group, but Theorem~\ref{thm:tessellations} produces non-congruent surfaces for dual tessellations.

The proof of Theorem~\ref{thm:tessellations} is inspired by the conjugate construction the authors~\cite{MT} performed for $H=\frac{1}{2}$ and $\epsilon=-1$. It is worth pointing out that some difficulties show up with respect to~\cite{MT}. First, we need to start from a minimal disk in $\mathrm{Nil}_3$ when $H$ is critical (i.e., $4H^2+\epsilon=0$), in a Berger spheres $\mathbb{S}^3_b$ when $H$ is supercritical (i.e., $4H^2+\epsilon>0$), or in the special linear group $\widetilde{\mathrm{Sl}}_2(\mathbb{R})$ when $H$ is subcritical (i.e., $4H^2+\epsilon<0$), whereas only the critical case was treated in~\cite{MT}. Second, we need to produce suitable barriers that guarantee that the $H$-surface has the desired symmetries, and study the different limit surfaces that appear in each range of $H$ separately. A brief introduction to $\mathbb{E}(\kappa,\tau)$-spaces and conjugate techniques will be given in Sections~\ref{sec:preliminaries} and~\ref{sec:fundamental-piece}, see also~\cite{Daniel07, HST,MPT,MT,Plehnert2}.

The second part of this paper deals with the embeddedness of $\Sigma_{H,k,m}$ in $\mathbb{S}^2\times\mathbb{R}$. As pointed out above, in $\mathbb{S}^2$ the value $H=\frac{1}{2}$ is not critical but has a special interpretation. In particular, for any $0<H<\frac{1}{2}$, we get the existence compact embedded $H$-surfaces in $\mathbb{S}^2\times\mathbb{R}$ with arbitrary genus, inducing compact non-orientable embedded $H$-surfaces in $\mathbb{R}\mathrm{P}^2\times\mathbb{R}$ with arbitrary even genus.

\begin{theorem}\label{thm:embeddedness}
Let $0 < H < \frac{1}{2}$ and $g\geq 2$. Then $\Sigma_{H,2,g+1}$ is a compact orientable $H$-bigraph embedded in $\mathbb{S}^2\times\mathbb{R}$ with genus $g$. Furthermore:
\begin{enumerate}
    \item[(a)] The limit of $\Sigma_{H,2,g+1}$ as $H\to\frac{1}{2}$ is a pair of $\frac{1}{2}$-spheres tangent along an equator; the limit of $\Sigma_{H,2,g+1}$ as $H\to 0$ is a double cover of $\mathbb{S}^2\times \{0\}$ with singularities at $g+1$ points evenly distributed along an equator.
  \item[(b)] If $g=2n-1$, then $\Sigma_{H,2,g+1}$ is antipodally invariant and induces a two-sided non-orientable $H$-surface of genus $2n$ in $\mathbb{R}\mathrm{P}^2\times\mathbb{R}$.
\end{enumerate}
\end{theorem}

\begin{figure}[hbp]
\centering
\includegraphics{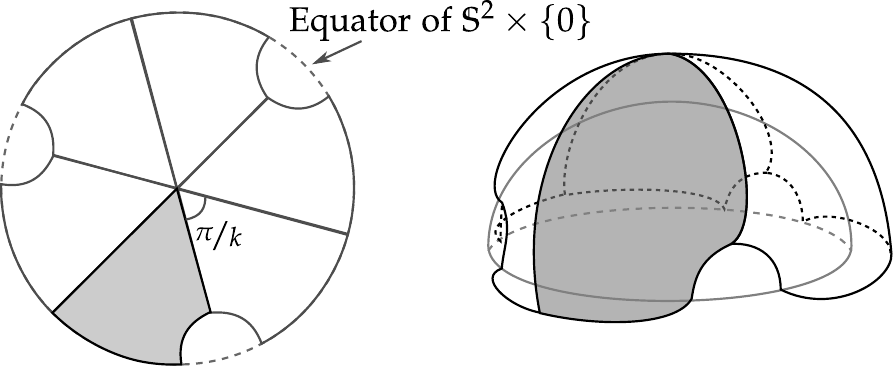}
\caption{Sketch of $\Sigma_{H,2,3}$ in the the model $\mathbb{R}^3\setminus\{0\}$ of $\mathbb{S}^2\times\mathbb{R}$ (right) and its top view (left). Observe that one needs $8(g+1)=24$ copies of the fundamental piece to get the complete bigraph.}
\label{fig:compact-surface}
\end{figure}

The case $g\leq 1$ is much simpler, since one can easily find rotationally invariant compact embedded $H$-surfaces in $\mathbb{S}^2\times\mathbb{R}$ with genus $0$ or $1$ (and arbitrary $H>0$), see~\cite{Manzano}. Pedrosa~\cite{Pedrosa} proved that isoperimetric regions in $\mathbb{S}^2\times\mathbb{R}$ are bounded by spheres, so the surfaces provided by Theorem~\ref{thm:embeddedness} cannot be isoperimetric. Other examples of non-rotationally invariant (and possibly non-embedded) tori in $\mathbb{S}^2\times\mathbb{R}$ were obtained by the authors in~\cite{MT}, some of which can be recovered as the degenerated case $\Sigma_{H,m,2}$ ($m\geq 3$) in Theorem~\ref{thm:tessellations}. Since any compact $H$-surface $\Sigma$ immersed in $\mathbb{S}^2\times\mathbb{R}$ must be orientable (if $H=0$, then $\Sigma$ must be a horizontal sphere by maximum principle; if $H>0$ then $\Sigma$ is trivially $2$-sided and hence orientable), Theorem~\ref{thm:embeddedness} implies the following nice statement:
\begin{quotation}
    \itshape
A compact surface can be embedded in $\mathbb{S}^2\times\mathbb{R}$ with constant mean curvature $0<H<\frac{1}{2}$ if and only if it is orientable.
\end{quotation}
Observe that any two compact embedded $H$-surfaces in $\mathbb{S}^2\times\mathbb{R}$ (for any $H>0$) with the same genus are isotopic, i.e., one can be deformed into the other by a continuous family of diffeomorphisms of $\mathbb{S}^2\times\mathbb{R}$ (see Remark~\ref{rmk:isotopy}). This is similar to Lawson's result for compact embedded minimal surfaces in $\mathbb{S}^3$, see~\cite{Law1970}.

The argument of the proof of Theorem~\ref{thm:embeddedness} consists in showing that the boundary of the fundamental piece of $\Sigma_{H,2,g+1}$ projects one-to-one to $\mathbb{S}^2$, and this relies on a subtle analysis of the lenght of the components of the projected boundary, as well as their geodesic curvature (by adapting the technique in~\cite{Manzano}). We remark that the same approach applies to the tessellations corresponding to the Platonic solids and the degenerated case $k=2$, so it is also possible to obtain compact embedded $H$-surfaces with $0<H<\frac{1}{2}$ in the cases listed in Table~\ref{table}. The hexahedron and octahedron also induce compact non-orientable embedded surfaces in $\mathbb{R}\mathrm{P}^2\times\mathbb{R}$ of genus $6$ and $8$, respectively. 

Let us finally discuss a couple of open questions. First, the embeddedness of our examples in $\mathbb{H}^2\times\mathbb{R}$ is not known (in $\mathbb{H}^2\times\mathbb{R}$ a Krust type theorem for $H$-surfaces is expected to hold true, as in the minimal case~\cite{HST}, which would prove our surfaces are embedded whenever it applies). Even in $\mathbb{R}^3$, the embeddedness of the surfaces has not been proved as far as we know. Second, our construction strongly depends on the size of spheres, which makes it possible only if $0<H<\frac{1}{2}$. It would be quite interesting to extend Theorem~\ref{thm:embeddedness} to any $H>0$.

\section{Preliminaries}\label{sec:preliminaries}

Given $\kappa,\tau\in\mathbb{R}$, let $\mathbb{M}^2(\kappa)$ be the simply connected surface with constant curvature $\kappa$ and let $\mathbb{E}(\kappa,\tau)$ be the simply-connected homogeneous three-manifold characterized by admitting a Killing submersion $\pi: \mathbb{E}(\kappa, \tau) \to \mathbb{M}^2(\kappa)$ with constant bundle curvature $\tau$ and constant Killing length over $\mathbb{M}^2(\kappa)$, see~\cite{Daniel07,LerMan}. The Killing submersion gives rise to natural notions of vertical and horizontal directions (as those tangent and orthogonal to the fibers of the submersion, respectively) which will be considered in the sequel.
    

The Daniel sister correspondence~\cite{Daniel07} establishes an isometric duality between minimal immersions $\widetilde\phi:\Sigma\to\mathbb{E}(4H^2+\epsilon,H)$, and $H$-immersions $\phi:\Sigma\to\mathbb{M}^2(\epsilon)\times\mathbb{R}$, where $H\in\mathbb{R}$ and $\epsilon\in\{-1,0,1\}$ are arbitrary, and $\Sigma$ stands for a simply connected Riemannian surface. It is well known that sister immersions enjoy the following properties (e.g., see~\cite{Daniel07,MT,Plehnert2}):
\begin{enumerate}[(i)]
  \item $\phi$ and $\widetilde{\phi}$ have the same \emph{angle function}, i.e., $\nu = \langle N,\xi\rangle = \langle\widetilde{N},\widetilde\xi\rangle$, where $N$ and $\widetilde{N}$ are appropriate unit normal vector fields to $\phi$ and $\widetilde\phi$, and $\xi$ and $\widetilde\xi$ denote the unit vertical Killing fields in $\mathbb{M}^2(\epsilon)\times\mathbb{R}$ and $\mathbb{E}(4H^2+\epsilon, H)$, respectively.

  \item The tangential projections $T = \xi - \nu N$ and $\widetilde{T} = \widetilde\xi - \nu \widetilde{N}$, as well as the shape operators $S$ and $\widetilde{S}$ of $\phi$ and $\widetilde{\phi}$, respectively, satisfy $\mathrm{d}\phi^{-1}(T) = J\mathrm{d}\widetilde\phi^{-1}(\widetilde{T})$ and $S = J \widetilde{S}$, where $J$ is a $\frac{\pi}{2}$-rotation in the tangent bundle of $\Sigma$.

  \item Any geodesic curvature line in the initial surface becomes a planar line of symmetry in the conjugate one: Given a curve $\Gamma\subset\Sigma$, then $\widetilde{\phi}(\Gamma)$ is a horizontal (resp.\ vertical) geodesic if and only if $\phi(\Gamma)$ lies in a vertical plane (resp.\ horizontal slice), which the immersion meets orthogonally.
\end{enumerate}
Throughout the rest of the paper, we will omit the immersions $\widetilde\phi$ and $\phi$, and refer to the (immersed) surfaces as $\widetilde\Sigma$ in $\mathbb{E}(4H^2+\epsilon,H)$ and $\Sigma$ in $\mathbb{M}^2(\epsilon)\times\mathbb{R}$.

\section{Proof of Theorem~\ref{thm:tessellations}}\label{sec:conjugate-Plateau-construction}

We will assume that $k\geq 3$ in the sequel (if $k=2$, it suffices to consider the $H$-tori constructed in~\cite{MT}). Although the argument below can be straightforwardly adapted to the case $\epsilon=0$, we will also assume $\epsilon\in\{-1,1\}$.

\subsection{The fundamental piece}\label{sec:fundamental-piece}
Let us consider a convex triangle $\Delta$ in $\mathbb{M}^2(4H^2+\epsilon)$ with three geodesic segments $\widetilde\beta_1,\widetilde\beta_2,\widetilde\beta_3$ as sides. The triangle $\Delta$ defines a geodesic quadrilateral $\widetilde\Gamma\subset\mathbb{E}(4H^2+\epsilon,H)$ with three horizontal sides $\widetilde h_1$, $\widetilde h_2$ and $\widetilde h_3$ such that $\pi(\widetilde h_i)=\widetilde \beta_i$, $i\in\{1,2,3\}$, and a vertical segment $\widetilde v$ projecting to the intersection of $\widetilde\beta_1$ and $\widetilde\beta_3$, which can be uniquely determined by the fact $\widetilde\Gamma$ is nullhomotopic in the mean-convex body $W=\pi^{-1}(\Delta)\subset\mathbb{E}(4H^2+\epsilon,H)$. Let us denote by the number $\widetilde 1$ (resp. $\widetilde 2$, $\widetilde 3$, $\widetilde 4$) the intersection of $\widetilde h_1$ and $\widetilde v$ (resp, $\widetilde h_1$ and $\widetilde h_2$, $\widetilde h_2$ and $\widetilde h_3$, $\widetilde h_3$ and $\widetilde v$), see Figure~\ref{fig:first-polygon} (left). It follows from Meeks and Yau's solution to the Plateau problem~\cite{MY82} that there is a unique area-minimizing embedded disk $\widetilde\Sigma\subset W$ with boundary $\widetilde\Gamma$.

By means of a standard argument consisting in a slight deformation of the boundary as in the proof of~\cite[Proposition 2]{MT}, it follows that interior of $\widetilde\Sigma$ is a graph over the interior of $\Delta$. Moreover, applying the boundary maximum principle, and intersecting $\widetilde\Sigma$ with appropriate umbrellas as in the proof of~\cite[Lemma 3]{MT}, we deduce that the angle function $\nu$ takes values in $[-1,0]$. Furthermore,
\begin{itemize}
  \item the only points on which $\nu$ vanishes are those in the vertical segment $\widetilde v$;
  \item the only points on which $\nu$ takes the value $-1$ are $\widetilde 2$ and $\widetilde 3$.
\end{itemize}
The sister $H$-surface $\Sigma$ in $\mathbb{M}^2(\epsilon)\times\mathbb{R}$ is bounded by curves $h_1$, $h_2$, $h_3$, and $v$, corresponding to $\widetilde h_1$, $\widetilde h_2$, $\widetilde h_3$ and $\widetilde v$, respectively, such that $h_1$, $h_2$ and $h_3$ are contained in vertical planes whereas $v$ lies in a horizontal slice, which will be assumed to be $\mathbb{M}^2(\epsilon)\times\{0\}$ (see Figure~\ref{fig:first-polygon}). Their endpoints will be denoted by 1--4, in correspondence with $\widetilde 1$--$\widetilde 4$. We will express $h_i=(\beta_i,z_i)\in\mathbb M^2(\epsilon)\times\mathbb{R}$, $i\in\{1,2,3\}$. Since $\Sigma$ intersects the plane containing $h_i$ orthogonally, it is not difficult to show that $\|\beta_i'\|=-\nu$ and $|z_i'|=\sqrt{1-\nu^2}$. From the above discussion about $\nu$, it follows that the curves $\beta_i$ are injective, and the height components $z_i$ are monotonic. 

\begin{figure}[htbp]
\centering
\includegraphics{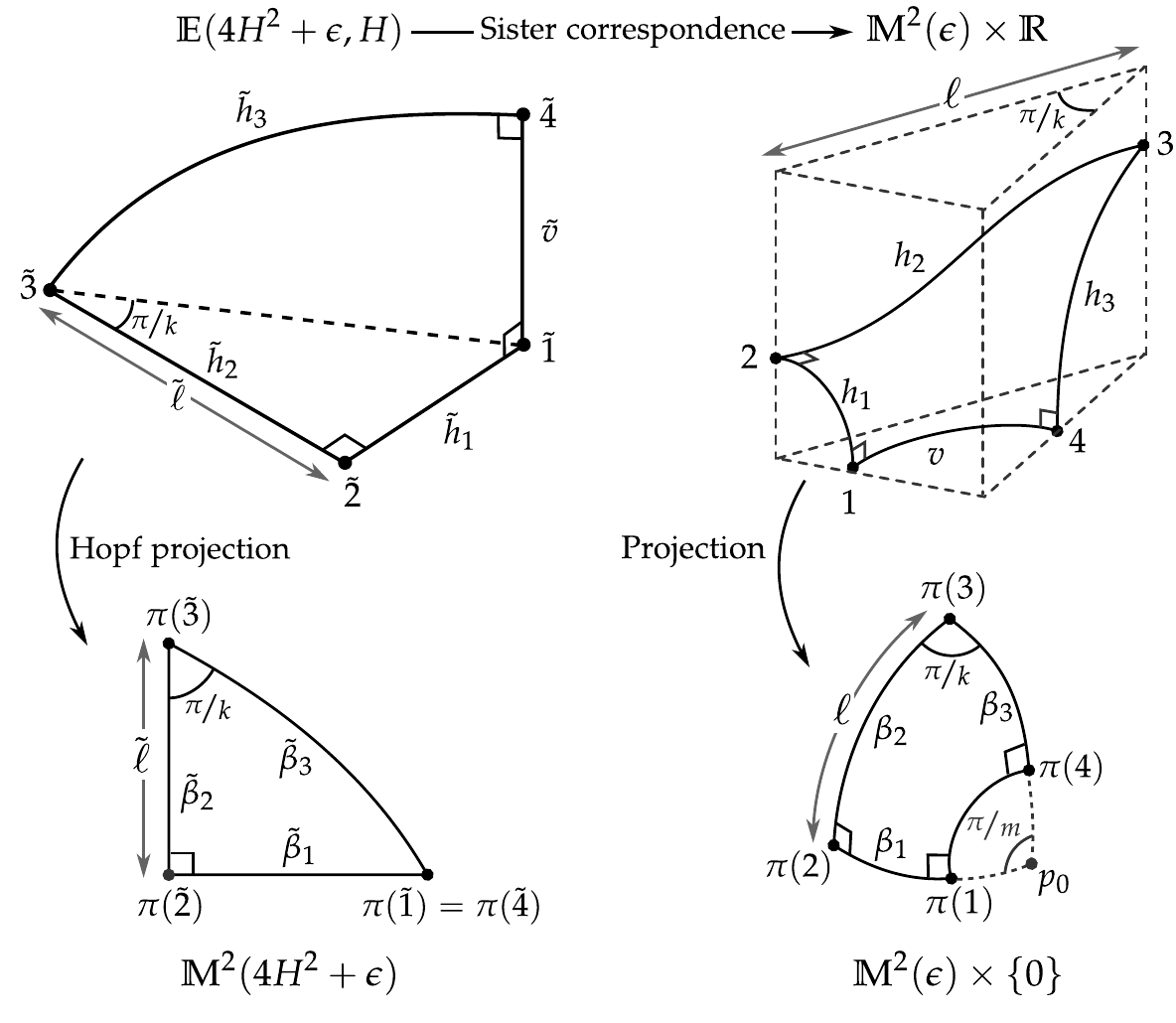}
\caption{Geodesic polygon $\widetilde\Gamma$ (top left), its sister $\Gamma$ (top right) for a suitable choice of the angles of $\Delta$ and the projection of them to $\mathbb{M}^2(4H^2+\epsilon)$ via Hopf fibration (bottom left) and $\mathbb{M}^2(\epsilon)$ (bottom right), respectively. Our target $\ell$ will be the length of the cathetus of a right triangle with angles $\pi/2$, $\pi/k$ and $\pi/m$ in $\mathbb{M}^2(\epsilon)$.}
\label{fig:first-polygon}
\end{figure}

The angles of $\Sigma$ at the vertexes $1$ and $4$ are equal to $\frac{\pi}{2}$, so we will get a complete smooth surface by means of successful axial symmetries about the geodesics of the boundary if and only if the angles at $2$ and $3$ are integer divisors of $\pi$ (due to the Schwarz reflection principle and the absence of isolated singularities, see~\cite[Sec.~2.2]{MT}). Such angles can be easily prescribed, since they are equal to those of $\widetilde\Sigma$ at $\widetilde 2$ and $\widetilde 3$ (which are in turn equal to those of $\Delta$ at $\pi(\widetilde 2)$ and $\pi(\widetilde 3)$). Furthermore, $\Sigma$ is horizontal at $2$ and $3$, and this means that the angle made by the vertical planes containing $h_1$ and $h_2$ (resp. $h_2$ and $h_3$) is the angle at $2$ (resp. 3), so we can easily understand the isometry group of the completion of $\Sigma$ provided that we control the distance between $\pi(2)$ and $\pi(3)$. To this end, we define the quantities
\begin{equation}\label{eqn:ell}
\begin{aligned}
\widetilde\ell&=\mathrm{Length}(\widetilde\beta_2)=\mathrm{Length}(\widetilde h_2),\\
\ell&=\mathrm{Length}(\beta_2)=-\int_{h_2}\nu=-\int_{\widetilde h_2}\nu.
\end{aligned}
\end{equation}
Our goal in the next subsection is to use different triangles $\Delta$ to produce $H$-surfaces $\Sigma$ such that the vertical planes containing the curves $h_i$ form a specific triangle that fits into a $(m, k)$-tessellation (see Figures~\ref{fig:4-gon-tessellation} and~\ref{fig:first-polygon}). Since two of its angles are determined, there is only one degree of freedom in the choice of $\Delta$ (namely, the value of $\widetilde\ell$), which will allow us to reach a specific value of $\ell$ via continuity arguments (observe that, because of the uniqueness of solution, the value of $\ell$ depends continuously on $\Delta$).

\subsection{Continuity arguments}\label{sec:continuity}
Let $m,k\geq 2$ be integers, let $H>0$, and let $\epsilon$ be the sign of $\tfrac{1}{k} + \tfrac{1}{m} - \tfrac{1}{2}$. We will consider the construction in Section~\ref{sec:fundamental-piece} with $\Delta\subset\mathbb{M}^2(4H^2+\epsilon)$ a triangle with angles $\frac{\pi}{2}$ at $\pi(\widetilde 2)$, and $\frac{\pi}{k}$ at $\pi(\widetilde 3)$. This produces convex triangles provided that $0<\widetilde\ell<\widetilde\ell_{\mathrm{limit}}$, where
\[
    \widetilde\ell_{\mathrm{limit}}=
    \begin{cases}
	\frac{\pi}{\sqrt{4H^2+\epsilon}},&\text{if }4H^2+\epsilon>0\text{ (supercritical)},\\
	\infty,&\text{if } 4H^2+\epsilon=0\text{ (critical)},\\
	\frac{\arctanh(\cos(\pi/k))}{\sqrt{-4H^2-\epsilon}},&\text{if } 4H^2+\epsilon<0\text{ (subcritical)}.\\
    \end{cases}
\]
This value of $\widetilde\ell_{\mathrm{limit}}$ represents the diameter of $\mathbb{S}^2(4H^2+\epsilon)$ in the supercritical case, and the cathetus of an ideal right triangle with angles $0$, $\frac{\pi}{2}$ and $\frac{\pi}{k}$ in $\mathbb{H}^2(4H^2+\epsilon)$ in the subcritical case.

Our target value of $\ell$ will be the length of the cathetus, adjacent to the angle $\frac{\pi}{k}$, of a right triangle with angles $\frac{\pi}{2}$, $\frac{\pi}{k}$ and $\frac{\pi}{m}$ in $\mathbb{M}^2(\epsilon)$ (see Figure~\ref{fig:4-gon-tessellation} and Figure~\ref{fig:first-polygon} bottom right). The explicit value of $\ell$ can be easily computed using the spherical or hyperbolic law of cosines as
\begin{equation}\label{eqn:ell-target}
\ell_{\mathrm{target}}=\begin{cases}
\arccos\left(\frac{\cos(\frac{\pi}{m})}{\sin(\frac{\pi}{k})}\right),&\text{if }\epsilon=1,\\
\arccosh\left(\frac{\cos(\frac{\pi}{m})}{\sin(\frac{\pi}{k})}\right),&\text{if }\epsilon=-1.
\end{cases}
\end{equation}
We will check that the degree of freedom given by $\widetilde\ell\in(0,\widetilde\ell_{\mathrm{limit}})$ allows us to reach the value $\ell=\ell_{\mathrm{target}}$ given by~\eqref{eqn:ell-target}. As $\widetilde\ell\to 0$, the minimal surface $\Sigma$ becomes very small, and so does the length of $\widetilde h_2$, so~\eqref{eqn:ell} implies that $\ell$ gets arbitrarily close to zero. Let us call $\ell_{\mathrm{limit}}$ the value of $\ell$ associated with $\widetilde\ell_{\mathrm{limit}}$, so the continuity of $\ell$ reduces the problem to proving that $\ell_{\mathrm{target}}<\ell_{\mathrm{limit}}$. This is achieved by studying the limit cases according to the sign of $4H^2+\epsilon$ (Cases A, B and C below).

We are done once we get the piece $\Sigma$ with $\ell=\ell_{\mathrm{target}}$, since the smooth compact surface $\Sigma_{H,m,k}$ is obtained by successive reflections of $\Sigma$ across its boundary.

\medskip
\noindent\textbf{Case A (subcritical mean curvature).} Assume that $4H^2+\epsilon<0$ in the above construction, so $\epsilon=-1$ and $\widetilde\Sigma$ lies in $\widetilde{\mathrm{Sl}}_2(\mathbb{R})$. We need to produce some barriers in order to give an appropriate bound for $\ell_{\mathrm{limit}}$.

\begin{lemma}\label{lemma:barrier1}
Let $\Gamma_1,\Gamma_2$ be two orthogonal horizontal geodesics contained in an umbrella $S_0\subset\mathbb{E}(\kappa,\tau)$, with $\kappa<0$ and $\tau\neq 0$. Then there is an entire minimal graph $S\subset\mathbb{E}(\kappa,\tau)$ such that $S\cap S_0=\Gamma_1\cup\Gamma_2$ with angle function $\nu_S\equiv -1$ along $\Gamma_2$.
\end{lemma}

\begin{proof}
Define the entire graph $S=\cup_{t\in\mathbb{R}}\phi_t(\Gamma_1)$, being $\{\phi_t:t\in\mathbb{R}\}$ the 1-parameter group of hyperbolic translations in $\mathbb{E}(\kappa,\tau)$ along $\Gamma_2$. Observe that $S$ is minimal since it is ruled by the horizontal geodesics $\phi_t(\Gamma_1)$ and invariant under axial symmetries about them. It is clear that $\Gamma_1\cup\Gamma_2\subset S\cap S_0$, and the other inclusion follows from the fact that if there existed $p\in(S\cap S_0)-(\Gamma_1\cup\Gamma_2)$, then there would be horizontal geodesics $\Gamma_3\subset S_0$ and $\Gamma_4\subset S$, both of them containing $p$, and $\Gamma_2\cup\Gamma_3\cup\Gamma_4$ would contain a closed horizontal geodesic triangle, in contradiction with the assumption $\tau\neq 0$. Finally, since $\Gamma_1\cup\Gamma_2\subset S\cap S_0$, we get $\nu_S=-1$ at $\Gamma_1\cap\Gamma_2$. This implies that $\nu_S\equiv -1$ along $\Gamma_2$ since $\phi_t$ are isometries preserving the angle function.
\end{proof}

\begin{lemma}\label{lemma:barrier2}
If $4H^2+\epsilon<0$, then 
\[\ell_{\mathrm{limit}}>\arccosh\left(\frac{1-4H^2\sin(\frac{\pi}{k})}{(1-4H^2)\sin(\frac{\pi}{k})}\right)>\ell_{\mathrm{target}}.\]
\end{lemma}

\begin{proof}
Let us consider the umbrella $T$ centered at $\widetilde 3$, and let $S$ be the surface given by Lemma~\ref{lemma:barrier1} when we choose the horizontal geodesics $\Gamma_1$ containing $\widetilde h_1$ and $\Gamma_2$ containing $\widetilde h_2$. Observe that $\widetilde\Sigma$ lies in a region bounded by $T$ and $S$:
\begin{itemize}
  \item $\Sigma$ lies below $T$: otherwise, the umbrella $T$ would contain some $p\in\widetilde h_1$, and hence also the horizontal segment joining $\widetilde 3$ and $p$, which produces a closed horizontal geodesic triangle with vertexes $\widetilde 2$, $\widetilde 3$, and $p$.
  \item $\Sigma$ lies above $S$: otherwise, $S$ would contain some $p\in\widetilde h_3$, and hence also a horizontal segment joining some $q\in\widetilde h_2$ and $p$, which produces a closed horizontal geodesic triangle with vertexes $\widetilde 3$, $p$ and $q$.
\end{itemize}
We deduce that the angle function of $\Sigma$ along $h_2=S\cap T\cap\Sigma$ can be compared with those of $S$ and $T$, yielding $-1=\nu_S<\nu_\Sigma<\nu_T$. Let $\gamma$ be a unit-speed horizontal geodesic with $\gamma(0)=\widetilde 3$, so a direct computation yields
\begin{align*}
\ell_{\mathrm{limit}}&>-\int_0^{\widetilde\ell_{\mathrm{limit}}}\nu_T(\gamma(t))\mathrm{d}t\\
&=\int_0^{\frac{\arctan(\cos(\pi/k))}{\sqrt{1-4H^2}}}\frac{\sqrt{2-8 H^2} \cosh \left(\frac{1}{2} \sqrt{1-4 H^2} t\right)}{\sqrt{\cosh \left(\sqrt{1-4 H^2} t\right)-8 H^2+1}}\mathrm{d}t.
\end{align*}
Evaluating this integral, we get the first inequality in the statement. As for the second one, in view of~\eqref{eqn:ell-target}, it is equivalent to proving the inequality 
\[1-\cos(\tfrac{\pi}{m})+4H^2(\cos(\tfrac{\pi}{m})-\sin(\tfrac{\pi}{k}))>0,\]
which follows from estimating separately $1-\cos(\tfrac{\pi}{m})>0$ and $\cos(\tfrac{\pi}{m})>\sin(\tfrac{\pi}{k})$ (in the latter we employ the fact that $\frac{\pi}{m}+\frac{\pi}{k}<\frac{\pi}{2}$).
\end{proof}

By the arguments at the beginning of Section~\ref{sec:continuity}, Lemma~\ref{lemma:barrier2} implies Theorem~\ref{thm:tessellations} in the subcritical case for any $H\in(0,\frac{1}{2})$. Let us analyze the limit cases. On the one hand, if $4H^2\to 0$, then the bundle curvature of the initial space $\mathbb{E}(4H^2+\epsilon,4H)$ tends to zero, so the length of $\widetilde v$ also tends to zero. The maximum principle implies that the minimal piece $\widetilde\Sigma$ lies in a smaller and smaller slab of $\mathbb{E}(4H^2+\epsilon,4H)$, so its angle function converges uniformly to $-1$ (this follows from a standard limit argument using the stability of $\Sigma$), so the angle function of the sister piece $\Sigma$ also converges uniformly to $-1$ and it becomes a slice (with a singularity at one of the vertexes) in the limit. On the other hand, since $\alpha(m,k)>\frac{1}{2}$ whenever $\epsilon=-1$, the limit $4H^2\to\alpha(m,k)$ does not occur in the subcritical case and will be analyzed in Case C below. 

\medskip
\noindent\textbf{Case B (critical mean curvature).} If $4H^2+\epsilon=0$, then $\epsilon=-1$ and $\widetilde\Sigma$ lies in $\mathrm{Nil}_3(\frac{1}{2})$. In~\cite[Lemma~5]{MT}, the authors proved that $\ell_{\mathrm{limit}}=\infty$ using a similar argument to that of Case A, i.e., an umbrella and a surface invariant under translations (which plays the role of the barrier given by Lemma~\ref{lemma:barrier1}) as barriers. Therefore Theorem~\ref{thm:tessellations} also holds true in this case (note that the limit cases $4H^2\to0$ and $4H^2\to\alpha(m,k)$ do not occur in the critical case).

\medskip
\noindent\textbf{Case C (supercritical mean curvature).} If $4H^2+\epsilon>0$, then the initial piece $\widetilde\Sigma$ lies in the Berger sphere $\mathbb{E}(4H^2+\epsilon,H)$. The following result reveals that $\ell_{\mathrm{limit}}$ can be computed explicitly, so we do not have to find any barriers.

\begin{lemma}\label{lemma:supercritical}
If $4H^2+\epsilon>0$, then
\[\ell_\mathrm{limit}=\begin{cases}
 2\arctan(\frac{1}{2H}),& \text{if }\epsilon=1,\\
 2\arctanh(\frac{1}{2H}),& \text{if }\epsilon=-1.\\
\end{cases}
\]
In particular, $\ell_\mathrm{target}<\ell_\mathrm{limit}$ if and only if $4H^2<\alpha(m,k)$ (see Theorem~\ref{thm:tessellations}).
\end{lemma}

\begin{proof}
If $\widetilde\ell\to\widetilde\ell_\mathrm{limit}$, which is nothing but the diameter of $\mathbb{E}(4H^2+\epsilon,H)$, then $\widetilde h_1$ eventually disappears, and $\widetilde\Sigma$ becomes $1/4k$ of a horizontal umbrella centered at $\widetilde 3$. Since umbrellas in $\mathbb{E}(4H^2+\epsilon,H)$ are minimal spheres,  it follows that $\Sigma$ is a sector of angle $\frac{\pi}{k}$ of the upper half of an $H$-sphere $S_H\subset\mathbb{M}^2(\epsilon)\times\mathbb{R}$. Hence $\ell$ is the radius of the spherical or hyperbolic circle over which the $S_H$ is a bigraph. The upper hemisphere of $S_H$ is the rotationally invariant surface given by
\[h(r)=\begin{cases}
\frac{4H}{\sqrt{4H^2+1}}\arccosh\left(\frac{\sqrt{4H^2+1}}{2H}\cos\left(\frac{r}{2}\right)\right),&\text{if }\epsilon=1,\\
\frac{4H}{\sqrt{4H^2-1}}\arccos\left(\frac{\sqrt{4H^2-1}}{2H}\cosh\left(\frac{r}{2}\right)\right),&\text{if }\epsilon=-1.
\end{cases}\]
where $h(r)$ denotes the height of $S_H$ at distance $r$ from a fixed point in $\mathbb{M}^2(\epsilon)$, see~\cite{Manzano}. Hence $\widetilde\ell_\mathrm{limit}$ is the value of $r$ such that $h(r)=0$, which is easily seen to coincide with the value given in the statement.

The last assertion in the statement is straightforward in view of~\eqref{eqn:ell-target}.
\end{proof}

Observe that Lemma~\ref{lemma:supercritical} proves Theorem~\ref{thm:tessellations} in the supercritical case. The limit of the construction when $4H^2\to\alpha(m,k)$, after reflection across the boundary, has been proved to be a stack of $H$-spheres whose centers are the vertexes of the $(m,k)$-tessellation of $\mathbb{M}^2(\epsilon)\times\mathbb{R}$ at height zero, and whose radii coincide with $\ell_\mathrm{limit}=\ell_\mathrm{target}$. If $\epsilon=1$, then the limit case $4H^2\to0$ also occurs in the subcritical case, but its discussion is analogous to that of case A.

\section{Proof of Theorem~\ref{thm:embeddedness}}

Let $0<H<\frac{1}{2}$ and $g\geq 2$, and consider the compact surface $\Sigma_{H,2,g+1}$ obtained by the conjugate technique in Section~\ref{sec:conjugate-Plateau-construction}. We have also shown that the limit when $H\to\alpha(2,g+1)=\frac{1}{2}$ is a stack of spheres, which is nothing but a pair of spheres tangent along a equator since $m=2$, and the limit when $H\to 0$ is a double cover of $\mathbb{S}^2\times\{0\}$ with singularities at the centers of the $g+1$ polygons, i.e., along $g+1$ points evenly distributed along an equator. 

Since $4(g+1)$ triangles congruent to $\Omega$ are needed to tessellate $\mathbb{S}^2$, and this only gives half of $\Sigma_{H,2,g+1}$, we deduce that we need $8(g+1)$ copies of $\Sigma$ to produce $\Sigma_{H,2,g+1}$. Gau\ss-Bonnet theorem ensures that the total curvature of each of the $8(g+1)$ pieces is
\[
  \int_\Sigma K = \frac{\pi}{g+1} - \frac{\pi}{2},
\]
so we easily deduce that the genus of $\Sigma_{H,2,g+1}$ is $g$. 

Observe that the symmetry group of the $(2,g+1)$-tessellation contains the antipodal map of $\mathbb{S}^2$ if and only if $g$ is odd. If $g=2n-1$, then it is very easy to find M\"obius strips in $\widehat\Sigma_{H,2,2n}$, the quotient surface in $\mathbb{R}\mathrm{P}^2\times\mathbb{R}$, so it is non-orientable. However, the unit normal to $\Sigma_{H,2,2n}$ induces a globally defined unit normal to $\widehat\Sigma_{H,2,2n}$, ane hence $\widehat\Sigma_{H,2,2n}$ is two-sided. Finally, since $4(g-1)=8n$ copies are needed to produce $\widehat\Sigma_{H,2,2n}$, its Euler characteristic is given by
\[2\pi\chi(\widehat\Sigma_{H,2,2n})=\int_{\widehat\Sigma_{H,2,2n}}K=8n\int_\Sigma K=8n\left(\frac{\pi}{2n}-\frac{\pi}{2}\right)=2\pi(2-2n),\]
so the non-orientable genus of $\widehat\Sigma_{H,2,2n}$ is nothing but $2-\chi(\hat\Sigma_{H,2,g+1})=2n$.

Therefore the only remaining part of the proof consists in showing that $\Sigma_{H,2,g+1}$ is embedded. It suffices to show that the associated fundamental piece $\Sigma$ is embedded and lies in $\Omega\times\mathbb{R}$, being $\Omega$ the triangle with sides $\pi/{2}$, $\pi/{m}={\pi}/{2}$ and $\pi/{k}=\pi/{(g+1)}$ at the vertexes $\pi(2)$, $\pi(3)$ and $p_0$, see Figure~\ref{fig:first-polygon}. Lemma~\ref{lemma:embeddedness2} will show that $\Sigma$ projects one-to-one into a subset of $\Omega$ (i.e., $\Sigma$ is a graph) and this will finish the proof.

\begin{lemma}\label{lemma:embeddedness1}
The spherical curve $\pi\circ v$ is strictly convex in $\mathbb{S}^2$ with respect to the outer conormal to $\pi(\Sigma)$ along its boundary.
\end{lemma}

\begin{proof}
Let $\Sigma_+=\Sigma_{H,2,g+1}\cap\{(p,t)\in\mathbb{S}^2\times\mathbb{R}:t\geq 0\}$, i.e., $\Sigma_+$ is the upper half of $\Sigma_{H,2,g+1}$ consisting of the $4(g+1)$ copies of $\Sigma$ produced by reflection across vertical planes (not using reflection across $\mathbb{S}^2\times\{0\}$). Note that both the height function $h$ (given by $h(p,t)=t$) and the angle function $\nu$ of $\Sigma_+$ vanish along $\partial\Sigma_+$, whereas $h$ is positive (by a simple application of the maximum principle) and $\nu$ is negative in the interior of $\Sigma_+$. We will briefly explain how the argument in~\cite[Corollary 3.5]{Manzano} is adapted to this situation to prove the statement. The function
\[\psi=h+\frac{4H^2}{\sqrt{1+4H^2}}\arctanh\left(\frac{\nu}{\sqrt{1+4H^2}}\right)\in C^\infty(\Sigma_+)\]
is well defined and subharmonic on the compact surface $\Sigma_+$ (with boundary). Since $\psi$ is not constant (it is constant if and only if $\Sigma_+$ is contained in an $H$-sphere), and vanishes along $\partial\Sigma_+$, we deduce that ${\partial\psi}/{\partial\eta}>0$ along $\partial\Sigma_+$, where $\eta=-\partial_t$ is the outer conormal vector field to $\Sigma_+$ along its boundary. If $0<H<\frac{1}{2}$, the above strict inequality implies that the curvature of $\partial\Sigma_+$ in $\mathbb{S}^2\equiv\mathbb{S}^2\times\{0\}$ is positive with respect to the outer conormal to $\pi(\Sigma)\subset\mathbb{S}^2$ along its boundary. Hence we are done since the curve $\pi\circ v$ lies in $\partial\Sigma_+$.
\end{proof}

\begin{lemma}\label{lemma:embeddedness2}
The fundamental piece $\Sigma$ of $\Sigma_{H,2,g+1}$ is a graph.
\end{lemma}

\begin{proof}
Let us begin by showing that the curves $\beta_i$ lie in the boundary of the triangle $\Omega$ (see Figure~\ref{fig:first-polygon}). In Section~\ref{sec:fundamental-piece}, we proved that they are injective curves, and $\beta_2$ coincides with a side of $\Omega$ by construction, so it suffices to show that $\pi(1)$ and $\pi(4)$ lie in $\Omega$. In the case of $\pi(1)$, this follows from the fact that the distance between $\pi(2)$ and $p_0$ is $\pi/k$, whereas the length of $\beta_1$ is shorter:
\[\mathrm{Length}(\beta_1)=-\int_{\widetilde{h}_1}\nu\leq\mathrm{Length}(\widetilde{\beta}_1)\leq\frac{\pi}{k\sqrt{4H^2+1}}<\frac{\pi}{k}.\]
Here we have used that the maximum length of $\widetilde\beta_1$ is attained when $\widetilde\beta_2$ is a quarter of a circumference of $\mathbb{S}^2(4H^2+1)$, as well as the fact that $H>0$. As for $\pi(4)$, since varying $\ell$ in the construction produces a foliation of a region of the Berger sphere $\mathbb{E}(4H^2+1,4H)$, this implies that $\nu$ along $\beta_3$ depends monotonically on $\ell$. The maximum value of $\mathrm{Length}(\beta_3)=-\int_{\widetilde{h}_3}\nu$ is attained when $\ell=\ell_{\mathrm{limit}}$, because increasing $\ell$ results in integrating a bigger $-\nu$ along a longer curve. In other words, the length of $\beta_3$ is less than $\pi/2$, the radius of the domain over which the $\frac{1}{2}$-sphere is a bigraph, so $\pi(4)$ lies in the boundary of $\Omega$ (the reader can find a similar argument in~\cite[Lemma~3]{MT}). 

Finally, let us deal with the curve $\pi\circ v$, which is strictly convex by Lemma~\ref{lemma:embeddedness1}, orthogonal to the sides of $\Omega$ at its endpoints, and lies in the interior of $\Omega$ around these endpoints. We claim that the interior of $\pi\circ v$ lies in the interior of $\Omega$. If it were not the case, the convexity ensures that $\pi\circ v$ would intersect itself or other points of a curve $\beta_i$. Again by convexity this would mean that $\pi(\partial\Sigma)$ has a loop in such a way it is not the boundary of an immersed domain of $\mathbb{S}^2$, i.e., $\pi(\Sigma)$ would not be an Alexandrov-embedded domain. This is a contradiction since $\pi_{|\Sigma}:\Sigma\to\pi(\Sigma)$ is an immersion (the angle function, which is equal to minus the Jacobian determinant of $\pi_{|\Sigma}$, does not vanish).
\end{proof}

\begin{remark}\label{rmk:polyhedra}
As pointed out in the introduction, embeddedness also holds in the case of other tessellations of $\mathbb{S}^2$ whenever $0<H<\frac{1}{2}$. We believe embeddedness still holds beyond $H=\frac{1}{2}$ or in $(m,k)$-tessellations with $\frac{1}{m}+\frac{1}{k}\leq\frac{1}{2}$, but the curve in the boundary fails to be convex in general, so our approach does not apply.
\end{remark}

\begin{remark}\label{rmk:isotopy}
For some genera, i.e., $g\in\{1, 3, 5, 11, 19\}$ (see Table~\ref{table}), we have constructed non-congruent constant mean curvature $H$-surfaces ($0<H<\frac{1}{2}$). However, any two compact embedded $H$-surfaces in $\mathbb{S}^2\times\mathbb{R}$ (for any $H>0$) with the same genus $g$ are isotopic. This is a consequence of the fact that such surfaces are symmetric bigraphs over domains of $\mathbb{S}^2$ whose complements are the union of $g+1$ open topological disks (these disks are convex whenever $0<H<\frac{1}{2}$) because of the Alexandrov reflection principle. Hence the original surfaces are isotopic because the corresponding domains of $\mathbb{S}^2$ are isotopic (they are homeomorphic to $\mathbb{R}^2$ minus $g$ points).
\end{remark}


\begin{thebibliography}{[19]}
  \bibliographystyle{alpha}

\bibitem{Daniel07}
B.~Daniel.
\newblock Isometric immersions into 3-dimensional homogeneous manifolds.
\newblock {\em Comment. Math. Helv.}, \textbf{82} (2007), no.~1, 87--131.

\bibitem{HTW}
D.~Hoffman, M.~Traizet, and B.~White.
\newblock Helicoidal minimal surfaces of prescribed genus.
\newblock \emph{Acta Math.}, \textbf{216} (2016), no.~2, 217--323.

\bibitem{KGB}
K.~Gro\ss e-Brauckmann.
\newblock New surfaces of constant mean curvature.
\newblock  \emph{Math.~Z.}, \textbf{214} (1993), no.~4, 527--565.

\bibitem{HST}
L.~Hauswirth, R.~Sa Earp, and E.~Toubiana.
\newblock Associate and conjugate minimal immersions in $M\times\mathbb{R}$.
\newblock \emph{Tohoku Math.~J.}, \textbf{60} (2008), no.~2, 267--286. 

\bibitem{Lawson}
H.~B.~Lawson.
\newblock Complete minimal surfaces in {$S\sp{3}$}.
\newblock \emph{Ann.\ of Math. (2)}, \textbf{92} (1970), 335--374.

\bibitem{Law1970}
H.~B.~Lawson. 
\newblock The unknottedness of minimal embeddings. 
\newblock \emph{Invent.~Math.}, \textbf{11} (1970), 183--187.

\bibitem{LerMan}
A.~M.~Lerma and J.~M.~Manzano. 
\newblock Compact stable surfaces with constant mean curvature in Killing submersions. 
\newblock \emph{Ann.~Mat.~Pura App.}, \textbf{196} (2017), no.~4, 1345–1364.

\bibitem{Manzano}
J.~M.~Manzano. 
\newblock Estimates for constant mean curvature graphs in $M\times\mathbb{R}$. 
\newblock \emph{Rev.~Mat.~Iberoam.}, \textbf{29} (2013), no.~2, 1263–1281.

\bibitem{MPT}
J.~M.~Manzano, J.\ Plehnert, and F.\ Torralbo.
\newblock Compact embedded minimal surfaces in $\mathbb{S}^2\times \mathbb{S}^1$.
\newblock \emph{Comm.~Anal.~Geom.}, \textbf{24} (2016), no.~2, 409--429

\bibitem{MT}
J.~M.~Manzano and F.~Torralbo.
\newblock New examples of constant mean curvature surfaces in $\mathbb{S}^2\times\mathbb{R}$ and $\mathbb{H}^2\times\mathbb{R}$.
\newblock \emph{Michigan Math.~J.}, \textbf{63} (2014), no.~4, 701--723.

\bibitem{MarNev}
F.~C.~Marques and A.~Neves.
\newblock Existence of infinitely many minimal hypersurfaces in positive Ricci curvature.
\newblock \emph{Invent.~Math.}, \textbf{209} (2017), no.~2, 577--616.

\bibitem{MY82}
W.~H.~Meeks and S.-T.~Yau.
\newblock The classical Plateau problem and the topology of three-dimensional manifolds.
\newblock \emph{Topology}, \textbf{21} (1982), no.~4, 409--442.

\bibitem{MSY}
W.~H.~Meeks, L.~Simon, and S.-T.~Yau.
\newblock Embedded minimal surfaces, exotic spheres, and manifolds with positive Ricci curvature.
\newblock \emph{Ann.~of Math.~(2)}, \textbf{116} (1983), no. 3, 621--659.

\bibitem{Pedrosa}
R.~Pedrosa.
\newblock The Isoperimetric Problem in Spherical Cylinders.
\newblock {\em Ann.~Glob.~Anal.~Geom.}, \textbf{26} (2004), no. 4, 333--354.

\bibitem{Plehnert2}
J.~Plehnert.
\newblock Constant mean curvature $k$-noids in homogeneous manifolds.
\newblock \emph{Illinois J.~Math.}, \textbf{58} (2014), no. 1, 233--249.

\bibitem{Tor2011} F.~Torralbo. 
\newblock Compact minimal surfaces in the Berger spheres. 
\newblock \emph{Ann. Global Anal. Geom.}, \textbf{41} (2011), no. 4, 391--405.
\end{thebibliography}
\end{document}